\documentclass[12pt]{article}%
   
\usepackage{amsmath,enumerate}
\usepackage{amsfonts}
\usepackage{amssymb}

\newtheorem{theorem}{Theorem}[section]

\newtheorem{lemma}[theorem]{Lemma}

\newtheorem{question}[theorem]{Question}

\newenvironment{proof}[1][Proof]{\noindent\textbf{#1.} }
{\hfill \ \rule{0.5em}{0.5em}}

\begin{document}


\title{\vspace{-1cm}On $r$-uniform linear hypergraphs with no Berge-$K_{2,t}$}
\author{Craig Timmons\thanks{Department of Mathematics and Statistics, California State University Sacramento.  
This work was supported by a grant from the Simons Foundation (\#359419).} }

\maketitle

\vspace{-1cm}

\begin{abstract}
Let $\mathcal{F}$ be an $r$-uniform hypergraph and $G$ be a multigraph.  
The hypergraph $\mathcal{F}$ is a \emph{Berge}-$G$
if there is a bijection $f: E(G) \rightarrow E( \mathcal{F} )$ such that $e \subseteq f(e)$ for each 
$e \in E(G)$.  
Given a family of multigraphs $\mathcal{G}$, a
hypergraph $\mathcal{H}$ is said to be $\mathcal{G}$-\emph{free} if for each $G \in \mathcal{G}$, 
$\mathcal{H}$ does not contain a subhypergraph that is isomorphic to a Berge-$G$.  
We prove bounds on the maximum number of edges in an $r$-uniform linear hypergraph that is 
$K_{2,t}$-free.  We also determine an asymptotic formula for the maximum number of edges 
in a linear 3-uniform 3-partite hypergraph that is $\{C_3 , K_{2,3} \}$-free.   
\end{abstract}


\section{Introduction}

Let $G$ be a multigraph and $\mathcal{F}$ be a hypergraph.
Following Gerbner and Palmer \cite{gp}, we say that 
$\mathcal{F}$ is a \emph{Berge}-$G$ if there is a bijection 
$f : E(G) \rightarrow E( \mathcal{F} )$ with the property that 
$e \subseteq f(e)$ for all $e \in E(G)$.  
This definition generalizes both Berge-cycles and Berge-paths in hypergraphs.   
Recall that for an integer $k \geq 2$, a Berge $k$-cycle is an alternating 
sequence $v_1 e_1 v_2 e_2 \cdots v_k e_k v_1$ of distinct vertices and edges such 
that $\{v_i , v_{i+1} \} \subseteq e_i $ for $1 \leq i \leq k - 1 $, and $\{v_k , v_1 \} \subseteq e_k$.  
A Berge $k$-path is defined in a similar way (omit $e_k$ and $v_1$ from the sequence).  
Given a family of multigraphs $\mathcal{G}$, the hypergraph $\mathcal{H}$ is 
\emph{$\mathcal{G}$-free} if 
for every $G \in \mathcal{G}$, the hypergraph $\mathcal{H}$ does not contain a subhypergraph that is 
isomorphic to a Berge-$G$.  Observe that Berge-$G$ is a family of hypergraphs.  For example, 
$\{ \{ a,b,c \} , \{c,d,e \} \}$ and $\{ \{ a , b, c \} , \{ b , c, d \} \}$ are non-isomorphic hypergraphs, but both 
are Berge-$G$'s where $G$ is the path whose edges are $\{b , c \}$ and $\{ c , d \}$.

Write $\textup{ex}_r ( n , \mathcal{G} )$ for the maximum number of edges 
in an $n$-vertex $r$-uniform hypergraph that is $\mathcal{G}$-free.  
The function $\textup{ex}_r (n , \mathcal{G} )$ is the \emph{Tur\'{a}n number} or 
\emph{extremal number} of $\mathcal{G}$.  
When $r = 2$ and $\mathcal{G}$ consists of simple graphs, $\textup{ex}_2 ( n , \mathcal{G} )$ coincides with the usual 
definition of Tur\'{a}n numbers.  
When $\mathcal{G} = \{G \}$, we write 
$\textup{ex}_r ( n , G)$ instead of $\textup{ex}_r ( n , \{ G \} )$.    

One of the most important results 
in graph theory is the so-called Erd\H{o}s-Stone-Simonovits Theorem which is a statement about Tur\'{a}n numbers of graphs.  

\begin{theorem}[Erd\H{o}s, Stone, Simonovits]\label{ess}
If $G$ is a graph with chromatic number $k \geq 2$, then 
\[
\textup{ex}_2 ( n , G) = \left( 1 - \frac{1}{k - 1} \right) \binom{n}{2} + o(n^2).
\]
\end{theorem}

Theorem \ref{ess} provides an asymptotic formula for the Tur\'{a}n number of any non-bipartite graph.  
No such result is known for $r \geq 3$ and in general, hypergraph Tur\'{a}n problems are considerably harder 
than graph Tur\'{a}n problems.  Despite this, there has been some success in estimating 
$\textup{ex}_r ( n , \mathcal{G} )$ when $\mathcal{G}$ contains short cycles.  
For instance, Bollob\'{a}s and Gy\"{o}ri \cite{bg} proved that
\[
\frac{1}{ 3 \sqrt{3} } n^{3/2} - o( n^{3/2} ) \leq 
\textup{ex}_3 ( n ,  C_5  )
\leq \sqrt{2} n^{3/2} + 4.5 n.
\]
In other words, the maximum number of triples in an $n$-vertex 3-uniform hypergraph with no Berge 5-cycle is $\Theta (n^{3 / 2} )$.
One of the motivations behind estimating $\textup{ex}_3 ( n , C_5 )$ 
is the problem of 
finding the maximum number of triangles in a graph with no 5-cycle.  We refer the reader to \cite{bg} and the papers of 
Gy\"{o}ri, Li \cite{gli}, and Alon and Shikhelman \cite{as} for more on the intriguing problem of finding the maximum number 
of copies of a graph $F$ in an $H$-free graph $G$.     

Lazebnik and Verstra\"{e}te \cite{lv}
proved several results concerning $r$-uniform hypergraphs that are $\{C_2 , C_3 , C_4 \}$-free.  
Here $C_2$ is the multigraph consisting of two parallel edges.  Recall that a hypergraph $\mathcal{F}$ is 
\emph{linear} if any two distinct edges of $\mathcal{F}$ intersect in at most one vertex.  It is easy to check that 
\begin{center}
a hypergraph is linear if and only if it is $C_2$-free.  
\end{center}
Lazebnik and Verstra\"{e}te showed that
\begin{equation}\label{lv bound}
\textup{ex}_3 ( n , \{C_2 , C_3 , C_4 \} ) = \frac{1}{6}  n^{3/2} + o(n^{3/2}).
\end{equation}
A consequence of this result is the asymptotic formula $T_3 (n , 8 , 4) = \frac{1}{6} n^{3/2} + o(n^{3/2})$ for the generalized 
Tur\'{a}n number $T_r (n, k , l )$.  This is defined to be the maximum number of edges in an $n$-vertex $r$-uniform 
hypergraph with the property that no $k$ vertices span $l$ or more edges.  
Provided cycles are defined in the Berge sense as above, one may say that a $\{C_2 , C_3 , C_4 \}$-free hypergraph 
is a hypergraph of girth 5, and this is the terminology that is used in \cite{lv}.  The interest in 
$\textup{ex}_3( n , \{C_2 , C_3 , C_4 \})$ has its origins in determining the maximum number of edges in a graph with girth 5 which is a 
well-known, unsolved problem of Erd\H{o}s (see (\ref{girth 5 bound}) below).  

For related results, including results for paths, cycles, and some general bounds, see \cite{gl}, \cite{fo}, and \cite {gp}, respectively.  
The case of cycles has received considerable attention.  
Collier-Cartaino, Graber, and Jiang \cite{cgj} investigated so-called linear cycles in linear hypergraphs.
Their paper has a particularly nice introduction that discusses several results in this area.   
Lastly, the papers of  Gy\"{o}ri and Lemons \cite{gl2, gl1, gl}, in which bounds 
on the number of edges in a hypergraph with no Berge $k$-cycle are obtained, are also important contributions.   

In this paper we consider what happens in (\ref{lv bound}) when $C_4$ is replaced by $K_{2,3}$.  Our main result is given in the following 
theorem.

\begin{theorem}\label{main thm}
For any integer $r \geq 3$, 
\[
\frac{1}{r^{3/2} } n^{3/2} - o( n^{3/2} ) 
\leq 
\textup{ex}_r ( n , \{C_2 , C_3 , K_{2 , 2r - 3 } \} ) 
\leq 
\frac{ \sqrt{ 2r - 4} }{ r ( r - 1) } n^{3/2} + \frac{n}{r}.
\]
\end{theorem} 

Since $\frac{1}{3^{3/2} } > \frac{1}{6}$, Theorem \ref{main thm} implies that there are 
3-uniform hypergraphs that are $\{C_2 , C_3 , K_{2,3} \}$-free and have more edges than any $\{C_2 , C_3 , C_4 \}$-free 3-uniform hypergraph.  For graphs, the best known bounds on the Tur\'{a}n number of $\{C_3 , C_4 \}$ are 
\begin{equation}\label{girth 5 bound}
\frac{1}{ 2 \sqrt{2} } n^{3/2} - o(n^{3/2} ) \leq 
\textup{ex}_2( n , \{C_3 , C_4 \} )
\leq 
\frac{1}{2} n^{3/2} + o(n^{3/2} ).
\end{equation}
In \cite{aksv} it is shown that 
$\textup{ex}_2(n , \{C_3 , K_{2,3} \} ) \geq \frac{1}{ \sqrt{3} } n^{3/2} - o(n^{3/2} )$.   Putting all of these results together, we see
that in both the graph case and the 
3-uniform hypergraph case, forbidding $K_{2,3}$ instead of $C_4$ allows one to have significantly more edges.  It is not known if this 
is also true for $r \geq 4$.  On an interesting related note, Erd\H{o}s has conjectured 
that the lower bound in (\ref{girth 5 bound}) is correct while in \cite{aksv} it is conjectured that the lower bound in 
(\ref{girth 5 bound}) can be improved.       
     
Our construction that establishes the lower bound in Theorem \ref{main thm} is $r$-partite.  In this case, the upper bound of 
Theorem \ref{main thm} can be improved by adapting the counting argument of \cite{lv} to the $K_{2,3}$-free case.    

\begin{theorem}\label{main thm2}
Let $r \geq 3$.  If $\mathcal{F}$ is a $\{C_2 , C_3 , K_{2,3} \}$-free $r$-uniform $r$-partite hypergraph with $n$ vertices in each part, then 
\[
| E( \mathcal{F} ) | \leq \sqrt{ \frac{2}{r  - 1} } n^{3/2} + n.
\]
Furthermore, for any $q$ that is a power of an odd prime, there is a 
3-uniform 3-partite $\{C_2 , C_3 , K_{2,3} \}$-free hypergraph with $q^2$ vertices in each part and $q^2 ( q - 1)$ edges.
\end{theorem}

A similar result for 3-uniform 3-partite $\{C_2 , C_3 , C_4 \}$-free graphs was proved in \cite{lv}.  
Let us write $z_r (n , \mathcal{G})$ for the maximum number of edges in a $\mathcal{G}$-free $r$-uniform $r$-partite hypergraph 
with $n$ vertices in each part.  Using this notation, we can state Theorem 2.6 of \cite{lv} as $z_3(n , \{C_2 , C_3 , C_4 \} ) \leq \frac{1}{ \sqrt{2} }n^{3/2} + n$ for all $n \geq 3$, and 
$z_3(n , \{ C_2 , C_3 , C_4 \} ) \geq \frac{1}{2}n^{3/2} - 3n$ for infinitely many $n$. 
Theorem \ref{main thm2} gives the asymptotic formula 
\[
z_3(n , \{C_2 , C_3 , K_{2,3} \} ) = n^{3/2} + o(n^{3/2}).
\]      

One drawback to Theorem \ref{main thm} is that the size of the forbidden graph $K_{2,2r-3}$ depends on $r$. 
There are two natural directions to pursue.  On one hand, we can fix $r$ and attempt to construct $K_{2,t}$-free hypergraphs where 
$t$ tends to infinity and at the same time, the number of edges increases with $t$.    
Our next theorem shows that this can be done at the cost of allowing $C_3$.  

\begin{theorem}\label{thm3}
Let $r \geq 3$ be an integer and $l$ be any integer with $2l + 1 \geq r$.
If $q \geq 2lr^3$ is a power of an odd prime and $n = rq^2$, 
then 
\[
\textup{ex}_r ( n , \{ C_2 , K_{2 , t + 1} \} ) \geq \frac{l }{r^{3/2}} n^{3/2} - \frac{l}{r} n 
\]
where $t = (r - 1)(2l^2  - l)$.
\end{theorem}     

The other direction is to fix $t$ and let $r$ become large.  This is a much more difficult problem as suggested by the results and 
discussion in \cite{lv}.  We were unable to answer the following slight variation of a question posed to us by Verstra\"{e}te \cite{jv}.  

\begin{question}
Is there a bipartite graph $F$ that contains a cycle for which the following holds: there is a positive integer $r(F)$ such that for all $r \geq r(F)$, we have 
\begin{equation}\label{q}
\textup{ex}_r ( n , \{  C_2 ,F \} ) = o ( \textup{ex}_2 ( n , F) ).
\end{equation}
\end{question}

Using the graph removal lemma, one can show that (\ref{q}) holds whenever $F$ is a non-bipartite graph provided $r \geq |V (F)|$.  
When $F= C_4$, the formula (\ref{lv bound}) implies that $\textup{ex}_3 ( n , \{C_2 , C_4 \} ) = \Omega ( \textup{ex}_2 ( n , C_4 ))$,
but it is not known if the same lower bound holds for larger $r$.  
Using blow ups of extremal graphs, Gerbner and Palmer \cite{gp} (see also \cite{gl2, gl} for cycles) proved that 
$\textup{ex}_r ( n , K_{s , t} ) = \Omega ( \textup{ex}_2 ( n , K_{s,t} ) )$ whenever $2 \leq r \leq s + t$, but the hypergraphs 
constructed using this method are not $C_2$-free.    
Improving
the lower bound 
on $\textup{ex}_3 ( n , \{C_2 , C_{2k} \} )$ that comes from random constructions is a problem that was mentioned 
explicitly by F\"{u}redi and \"{O}zkahya in \cite{fo}.

In the next section we prove the upper bounds stated in Theorems \ref{main thm} and 
\ref{main thm2}.  Both of these upper bounds use the counting arguments of \cite{lv}. We include their proofs for completeness, but we 
do want to make it clear that proving our upper bounds using the methods of \cite{lv} is straightforward.  
The lower bounds of Theorems \ref{main thm}, \ref{main thm2}, and \ref{thm3} are our main contribution.  
Section \ref{algebra} contains algebraic lemmas which are required for our construction.
Section \ref{the H} gives the construction which is a generalization of the one found in \cite{tv} and is 
based on a construction Allen, Keevash, Sudakov, and Verstra\"{e}te (see Theorem 1.6 \cite{aksv}).       


\section{Upper bounds}


\subsection{The upper bound of Theorem 1.2}

Using the counting argument of \cite{lv} we can prove an upper bound on the 
number of edges in a $\{ C_2  , C_3 , K_{2,t+1} \}$-free $r$-uniform hypergraph.  Given a set $S$, write 
$S^{(2)}$ for the set of pairs of elements of $S$.  In this section we prove the following which implies 
the upper bound given in Theorem \ref{main thm}.    

\begin{theorem}\label{ub}
If $r \geq 3$ and $t \geq 1$ are integers, then 
\[
\textup{ex}_r ( n , \{ C_2 , C_3 , K_{2,t+1} \}  ) \leq \frac{ \sqrt{t} }{r(r-1) } n^{3/2} + \frac{n}{r}.
\]
\end{theorem}
\begin{proof}
Let $\mathcal{F}$ be a $\{ C_2 , C_3 , K_{2,t+1} \}$-free $r$-uniform hypergraph with $n$ vertices.  
Let $V$ be the vertex set of $\mathcal{F}$.  
For $v \in V$, let $e_1^v , \dots , e_{d(v)}^{v}$ be the edges in $\mathcal{F}$ 
that contain $v$ where $d(v)$ is the degree of $v$ in $\mathcal{F}$.  
For $1 \leq i < j \leq d(v)$, let 
\[
P( e_i^v , e_j^v ) = \left\{  \{ x,y \} \in V^{(2)} : x \in e_i^v \backslash \{v\} ~\mbox{and}~ y \in e_j^v \backslash \{v \} \right\}.
\]
Since $\mathcal{F}$ is linear, the sets 
$e_1^v \backslash \{ v \} , e_2^v \backslash \{v\} , \dots , e_{d(v)}^v \backslash \{v \}$ are pairwise disjoint so 
we have $| P ( e_i^v , e_j^v ) | = (r-1)^2$.  For any fixed vertex $v$, 
\begin{equation}\label{ub eq1}
\sum_{1 \leq i < j \leq d(v) } | P ( e_i^v , e_j^v ) | = (r-1)^2 \binom{ d(v) }{2}
\end{equation}
and the sum in (\ref{ub eq1}) never counts a pair $\{x,y \} \in V^{(2)}$ more than once.  

Now consider the sum 
\begin{equation}\label{ub eq2}
\sum_{v \in V} \sum_{1 \leq i < j \leq d(v) } | P( e_i^v , e_j^v ) |.
\end{equation}
Suppose a pair $\{x,y \} \in V^{(2)}$ is counted more than $t$ times in this sum.  
Let $v_1, \dots , v_{t+1}$ be distinct 
vertices such that there are edges $e_i \neq f_i \in E ( \mathcal{F})$, both of which contain $v_i$, and 
$\{x , y \} \in P ( e_i , f_i )$ for $1 \leq i \leq t+1$.  Assume $x \in e_i$ and $y  \in f_i$.  
By definition of $P(e,f)$, $\{x,y \} \cap \{v_1 , \dots , v_{t+1} \} = \emptyset$ so  
$x,y , v_1 , \dots , v_{t+1}$ are all distinct.  If $e_1, \dots , e_{t+1}$, $f_1 , \dots , f_{t+1}$ are all distinct, then 
$\mathcal{F}$ contains a $K_{2,t+1}$ so these $2t+2$ edges cannot all be distinct.  We will show that this leads to a contradiction.  

If $e_i = e_j$ for some $1 \leq i < j \leq t+1$, then $v_j \in e_i$ and 
$\{ f_i , f_j , e_i \}$ is a $C_3$ since $v_i \in e_i \cap f_i$, $y \in f_i \cap f_j$, and 
$v_j \in f_j \cap e_i$.  Note that $f_i \neq f_j$ otherwise 
$\{v_i , v_j \} \subseteq f_i \cap e_i$ contradicting the linearity of $\mathcal{F}$. 
We conclude that $e_i \neq e_j$ for $1 \leq i < j \leq t + 1$.   
A similar argument shows that $f_i \neq f_j$ for $1 < i < j \leq t+1$.  The only remaining possibility is that 
$e_i = f_j$ for some $1 \leq i \neq j \leq t+1$.  
If this is the case, then $y \in e_i$ so $\{v_i , y \} \subseteq e_i \cap f_i$ which, by linearity, implies $e_i = f_i$ which 
is a contradiction.    

We conclude that the sum (\ref{ub eq2}) counts any pair 
$\{x,y \} \in V^{(2)}$ at most $t$ times.  Let $m$ be the number of edges of $\mathcal{F}$.  
By (\ref{ub eq1}) and Jensen's Inequality
applied to the convex function 
\[
f(x) = 
\left\{
\begin{array}{ll}
\binom{x}{2} & \mbox{if $x \geq 2$} \\
0 & \mbox{otherwise},
\end{array}
\right.
\]
we have
\[
t \binom{n}{2} 
\geq 
\sum_{v \in V} \sum_{1 \leq i < j \leq d(v) } | P ( e_i^v , e_j^v ) |
\geq
 (r-1)^2 \sum_{v \in V} \binom{d(v)}{2} \geq n (r-1)^2 \binom{rm/n}{2}.
\]
This is a quadratic inequality in $m$ and implies that 
\[
m \leq \left(  \frac{t n^3}{r^2 ( r - 1)^2 } + \frac{n^2}{4r^2} \right)^{1/2} + \frac{n}{2r} \leq \frac{ \sqrt{t} }{r(r-1)} 
n^{3/2}+  \frac{n}{r}.
\]
\end{proof}


\subsection{The upper bound of Theorem 1.3}

The upper bound of Theorem \ref{main thm2} essentially follows from Theorem 2.3 in \cite{lv} with some modifications to the 
proof.  We include the proof for completeness.

\begin{theorem}
Let $r \geq 3$.  If $\mathcal{F}$ is a $\{C_2 , C_3 , K_{2,3} \}$-free $r$-uniform $r$-partite hypergraph with $n$ vertices in each part, then 
\[
| E( \mathcal{F} ) | \leq  \sqrt{ \frac{2}{r - 1} } n^{3/2} + n.
\]
\end{theorem}
\begin{proof}
Let $\mathcal{F}$ be an $r$-partite $r$-uniform hypergraph with $n$ vertices in each part. 
Let $V_1 , \dots , V_r$ be the parts of $\mathcal{F}$ and assume that $\mathcal{F}$ is $\{ C_2 , C_3 , K_{2,3} \}$-free.  
Let $S$ be the set of all pairs of the form 
$(v , \{x,y \})$ where $v \in V( \mathcal{F} )$, $\{x ,y \}$ is a pair of vertices in the same part with $x \neq v$, $y \neq v$, and 
there are distinct edges $e$ and $f$ with $\{v , x \} \subset e$ and $\{v , y \} \subset f$.  We will count the 
cardinality of $S$ in two ways.  Given a vertex $v \in \mathcal{F}$, we again write $d(v)$ for the number of edges that 
contain $v$.    

If we first choose the vertex $v$, there are $\binom{d(v)}{2} (r - 1)$ ways to choose a pair $\{x , y \}$ for 
which $(v , \{x,y \} )$ belongs to $S$.  Here we are using the fact that $\mathcal{F}$ is linear and so 
every edge of $\mathcal{F}$ contains 
exactly one vertex in each part.  Therefore, 
\begin{equation}\label{main thm2 eq1}
|S| = 
 \sum_{v \in V( \mathcal{F} ) } \binom{ d(v) }{2} (r -1) 
= 
(r - 1) \sum_{i = 1}^{r} \sum_{ v \in V_i} \binom{ d(v) }{2}.
\end{equation}    

Next we show that 
\begin{equation}\label{main thm2 eq2}
|S| \leq 2 \sum_{i = 1}^{r} \binom{ |V_i | }{2}.
\end{equation}
We first pick a pair $\{x , y \}$ that are in the same part, say $\{ x , y \} \subset V_i$.  
We now claim that there are at most two distinct $v$'s for which $(v , \{ x , y \} )$ belongs to $S$.
Aiming for a contradiction, suppose that $(v , \{x , y \} )$, $(v' , \{x , y \} )$, and 
$(v '' , \{x , y \} )$ all belong to $S$ where $v,v'$, and $v''$ are all distinct.  
Let $e$, $e'$, and $e''$ be the edges through $x$ that 
contain $v$, $v'$, and $v''$, respectively.   
Let $f , f'$, and $f''$ be the edges through $y$ that contain $v$, $v'$, and $v''$, respectively.  

If $v$, $v'$, and $v''$ are all in the same part, then $e,e',e'',f,f'$, and $f''$ are all distinct and we have 
a $K_{2,3}$.  
Therefore, we can assume at least two of $v$, $v'$, and $v''$ are in different parts.  

Suppose that $v \in V_j$ and $v ' \in V_k$ where $i$, $j$, and $k$ are all distinct.
If $v' \in e$, then $e$, $f'$, and $f$ form a $C_3$ in $\mathcal{F}$
since 
$v' \in e \cap f'$, $y \in f' \cap f$, and $v \in f \cap e$.  
Also note that the edges $e$, $f$, and $f'$ are all distinct since by definition, $e$ and $f$ are distinct edges, and $f'$ cannot 
be $e$ since $f' \cap V_i = \{y \}$ but $e \cap V_i = \{ x \}$.  Lastly, $f$ cannot be $f'$ otherwise 
$\{ v , v' \} \subset e \cap f$ which, by linearity, would imply $e = f$, a contradiction.

Combining (\ref{main thm2 eq1}) and (\ref{main thm2 eq2}) and using the fact that $|V_i| = n$ for every $i$, we have 
\begin{equation*}
2 r \binom{n}{2} = 2 \sum_{i = 1}^{r} \binom{ |V_i | } {2} \geq |S| = (r - 1) \sum_{i = 1}^{r} \sum_{v \in V_i} \binom{ d(v) }{2}.
\end{equation*}
By Jensen's Inequality, $\sum_{v \in V_i} \binom{ d(v) }{2} \geq n \binom{ m / n }{2}$ where $m$ is the 
number of edges of $\mathcal{F}$.  Together, these two estimates give $2r \binom{n}{2} \geq (r - 1 ) rn \binom{ m / n }{2}$ so 
\begin{equation*}
r n ( n - 1) \geq  (r - 1)r  n  \frac{ (m/n) (m/n - 1) }{2}. 
\end{equation*}
It follows that 
\[
m \leq \sqrt{ \frac{2}{r - 1} } n^{3/2} + n.
\]
\end{proof}


\section{Lower bounds}

In this section we prove the lower bounds of Theorems \ref{main thm}, \ref{main thm2}, and \ref{thm3}.  


\subsection{Algebraic Lemmas}\label{algebra}

In this subsection we prove some lemmas that are needed to prove our lower bounds.  
We write $\mathbb{F}_q$ for the finite field with $q$ elements and $\mathbb{F}_q^*$ for the 
group $\mathbb{F}_q \backslash \{ 0 \}$ under multiplication.  

The first lemma is due to Ruzsa \cite{ruzsa} and was key to the construction in \cite{tv}.  A proof can be found in \cite{tv}.    

\begin{lemma}\label{ruzsa}
Suppose $\alpha , \beta , \gamma$, and $\delta$ are nonzero elements of $\mathbb{F}_q$ with 
$\alpha + \beta = \gamma + \delta$.  If $a_1 , a_2 , a_3 , a_4 \in \mathbb{F}_q^*$, 
$\alpha a_1 + \beta a_2 = \gamma a_3 + \delta a_4$, and 
$\alpha a_1^2 + \beta a_2^2 = \gamma a_3^2 + \delta a_4^2$, 
then
\[
\alpha \beta ( a_1 - a_2 )^2 = \gamma \delta (a_3 - a_4)^2.
\]  
\end{lemma}

\bigskip

The next lemma is known.  It is merely asserting the well-known fact that $\{ (a,a^2) : a \in \mathbb{F}_q^* \}$ is a Sidon 
set in the group $\mathbb{F}_q \times \mathbb{F}_q$ where the group operation is componentwise addition.  

\begin{lemma}\label{arithmetic l2}
If $a_1 , a_2 , a_3 , a_4 \in \mathbb{F}_q^*$, $a_1 + a_2 = a_3 + a_4$, and 
$a_1^2 + a_2^2 = a_3^2 + a_4^2$, then $\{a_1 , a_2 \} = \{ a_3 ,a_4 \}$.  
\end{lemma}

The next two lemmas will be used to control the appearance of small graphs in our construction.  The idea 
is that a copy of some small graph in our construction corresponds to a nontrivial solution 
to some system of equations over $\mathbb{F}_q$.  
Variations of these lemmas have appeared in \cite{tv}.  

\begin{lemma}\label{arithmetic l1}
Let $\alpha , \beta$, and $\gamma$ be distinct elements of $\mathbb{F}_q$.  If $a_1 , a_2 , a_3 \in \mathbb{F}_q^*$, 
\begin{equation}\label{arithmetic l1 eq1}
0  = \alpha ( a_2 - a_1) + \beta ( a_3- a_2) + \gamma ( a_1 - a_3),
\end{equation}
and
\begin{equation*}
0 = \alpha ( a_2^2 - a_1^2 ) + \beta (a_3^2 - a_2^2 ) + \gamma (a_1^2 - a_3^2 ),
\end{equation*}
then $a_1 =a_2 = a_3$.
\end{lemma}
\begin{proof}
Adding $\beta a_1$ to both sides of (\ref{arithmetic l1 eq1}) and rearranging gives
\begin{equation}\label{arithmetic l1 eq3}
( \gamma - \beta )(a_3 - a_1) = ( \alpha  - \beta )(a_2 - a_1).
\end{equation}
A similar manipulation yields $( \gamma - \beta )(a_3^2 - a_1^2 ) = ( \alpha - \beta )(a_2^2 - a_1^2 )$ which is equivalent to  
\begin{equation}\label{arithmetic l1 eq4}
( \gamma - \beta )(a_3 - a_1) (a_3 + a_1) 
= ( \alpha - \beta )( a_2 - a_1)(a_2 + a_1).
\end{equation}
Note that $\gamma - \beta \neq 0$ and $\alpha - \beta \neq 0$ since $\alpha , \beta$, and $\gamma$ are all different.
If $a _ 3 = a_1$, then (\ref{arithmetic l1 eq3}) implies that $a_2 = a_1$ and we are done.
Otherwise, we divide (\ref{arithmetic l1 eq4}) by (\ref{arithmetic l1 eq3}) to get $a_3 + a_1 = a_2 + a_1$ which gives 
$a_3 = a_2$.  This equality, together with (\ref{arithmetic l1 eq1}), implies 
$0 = \alpha ( a_2 - a_1 ) + \gamma ( a_1 - a_2)$ so 
\[
\gamma (a_2 - a_1) = \alpha (a_2 - a_1).
\]
If $a_2 - a_1 = 0$, then with $a_3 = a_2$ we get $a_1 = a_2 = a_3$ and we are done.  Otherwise, we may cancel $a_2 - a_1$ 
to get $\gamma = \alpha$ which contradicts the fact that $ \gamma \neq \alpha$.      
\end{proof}

\begin{lemma}\label{arithmetic l3}
Let $\alpha  , \beta \in \mathbb{F}_q^*$ with $\alpha + \beta \neq 0$.  If $a_1, a_2,a_3 , b_1 , b_2 , b_3 \in \mathbb{F}_q^*$,
\begin{equation}\label{arithmetic l3 eq1}
\alpha a_1 +  \beta  b_1 = \alpha a_2 +  \beta b_2 = 
 \alpha  a_3 +  \beta b_3,
\end{equation}
and
\begin{equation*}\label{arithmetic l3 eq2}
 \alpha a_1^2 +  \beta  b_1^2 =  \alpha  a_2^2 +  \beta  b_2^2 = 
 \alpha  a_3^2 +  \beta b_3^2,
\end{equation*}
then there is a pair $\{i ,j \} \subset \{1,2,3 \}$ with $a_i = b_i$ and $a_j = b_j$.  
\end{lemma}
\begin{proof}
By Lemma \ref{ruzsa}, 
\begin{equation}\label{arithmetic l3 eq3}
 \alpha  \beta ( a_1  -b_1)^2 =  \alpha  \beta (a_2 - b_2)^2.
\end{equation}
Since $\alpha  \beta \neq 0$, (\ref{arithmetic l3 eq3}) implies that 
$(a_1 - b_1)^2 = (a_2 - b_2)^2$ so either $a_1 - b_1 = a_2 - b_2$, or $a_1 - b_1 = b_2 - a_2$.  

Suppose $a_1 - b_1 = a_2 - b_2$.  We multiply this equation through by $\alpha $ and subtract the resulting equation 
from the first equation in (\ref{arithmetic l3 eq1}) to get
\[
( \alpha + \beta ) b_1 = ( \alpha + \beta )b_2 .
\]
As $\alpha + \beta \neq 0$, it must be the case that $b_1 = b_2$ which, with (\ref{arithmetic l3 eq1}), gives 
$a_1 = a_2$ and we are done.

Now suppose that $a_1  -b_1  = b_2 - a_2$.  By symmetry, we may then assume that 
$a_1 - b_1 = b_3 - a_3$.  We then have $a_2 - b_2 = a_3 - b_3$ and the argument 
from the previous paragraph gives $a_2 = a_3$ and $b_2 = b_3$.    
\end{proof}


\subsection{The Construction}\label{the H}

Let $r \geq 2$ and $l \geq 1$ be integers.  Let $q$ be a power of an odd prime.  
Let $\alpha_1 , \dots , \alpha_r$ be distinct elements of 
$\mathbb{F}_q$.  We choose $q$ large enough so that there are distinct 
elements $m_1 , \dots  , m_l \in \mathbb{F}_q^*$ that satisfy the condition 
\begin{equation}\label{condition}
m_s ( \alpha_k - \alpha_i ) \neq m_t ( \alpha_k - \alpha_j)
\end{equation}
whenever $1 \leq s , t \leq l$ and $i,j$, and $k$ are distinct integers with $1 \leq i , j , k \leq r$.     

For $1 \leq i \leq r$, let $V_i = \mathbb{F}_q \times \mathbb{F}_q \times \{ i \}$.  
The union $V_1 \cup V_2 \cup \dots \cup V_r$ will be the vertex set of our hypergraph.  
We now define the edges.  Each edge will contain exactly one element from each $V_i$.    
Given $x,y \in \mathbb{F}_q$, $a \in \mathbb{F}_q^*$, and an integer $s \in \{1,2, \dots , l \}$, let 
\begin{eqnarray*}
e( x, y , a , m_s ) & = & \{
( x  + \alpha_1 (m_s a)  , y + \alpha_1 (m_s  a^2 ) , 1 ) , 
(x + \alpha_2 (m_s a ), y  + \alpha_2 (m_s  a^2 ) , 2) , 
\\
& ~ &
\dots , 
(x + \alpha_r (m_s   a ) , y + \alpha_r ( m_s a^2 ) , r) \}.
\end{eqnarray*}
We define $\mathcal{H}$ to be the $r$-uniform hypergraph with vertex set 
\[
V( \mathcal{H} ) = \{ ( x , y , i ) : x , y \in \mathbb{F}_q , 1 \leq i \leq r \}
\]
and edge set 
\[
E( \mathcal{H} ) = \{ e( x , y , a , m_s ) : x, y \in \mathbb{F}_q , a \in \mathbb{F}_q^* , s \in \{1, \dots l \}  \}.
\]
The vertex set of $\mathcal{H}$ can be written as $V( \mathcal{H} ) = V_1 \cup \dots \cup V_r$ so $\mathcal{H}$ is $r$-partite.  

\begin{lemma}\label{super linear}
The hypergraph $\mathcal{H}$ is linear.
\end{lemma}
\begin{proof}
Suppose $e(x_1 , y_1 , a_1 , m_s)$ and $e(x_2 , y_2 , a_2 , m_t )$ 
are edges of $\mathcal{H}$ that 
share at least two vertices, say $(u_i , v_i , i)$ in $V_i$ and 
$(u_j , v_j , j)$ in $V_j$, where $1 \leq i < j \leq r$.  
We have
\begin{center}
$u_i = x_1 + \alpha_i ( m_s  a_1 ) = 
x_2 + \alpha_i ( m_t  a_2 )$, ~~$v_i = y_1 + \alpha_i ( m_s  a_1^2 ) = y_2 + \alpha_i ( m_t   a_2^2)$,

\medskip

$u_j = x_1 +  \alpha_j (m_s a_1 ) = x_2 +  \alpha_j (m_t a_2 )$, ~~$v_j = y_1 +  \alpha_j (m_s a_1^2 ) 
= y_2 +   \alpha_j (m_t a_2^2)$.
\end{center} 
Taking differences yields
\begin{equation*}\label{super linear eq1}
u_i - u_j = m_s a_1 ( \alpha_i - \alpha_j ) = m_t a_2 ( \alpha_i - \alpha_j)
\end{equation*}
and
\begin{equation*}\label{super linear eq2}
v_i - v_j = m_s a_1^2 ( \alpha_i - \alpha_j ) = m_t a_2^2 ( \alpha_i - \alpha_j).
\end{equation*}
Since $\alpha_i$ and $\alpha_j$ are distinct, we may cancel $\alpha_i - \alpha_j$ to obtain 
$m_s a_1 = m_t a_2$ and $m_s a_1^2 = m_t a_2^2$.  All of the elements $m_s , m_t , a_1$, and $a_2$ are not zero so that 
this pair of equations implies that $a_1 = a_2$ and $m_s = m_t$.  It then follows from 
$x_1 +  \alpha_i (m_s a_1 ) = x_2 +  \alpha_i (m_t a_2 )$ that $x_1 = x_2$ and similarl,y $y_1 = y_2$.  We conclude that 
$e( x_1 , y_1 , a_1 , m_s) = e( x_2 , y_2 , a_2 , m_t)$ and so $\mathcal{H}$ is linear.  
\end{proof}

\bigskip

From Lemma \ref{super linear} we see that $\mathcal{H}$ has $l q^2 (q - 1)$ edges and it is clear that $\mathcal{H}$ has  
$r q^2$ vertices.  When $r = 2$, $\mathcal{H}$ is a graph.  

\bigskip

\noindent
\textbf{Example}~ Let $r = 2$, $l = 1$, $q \geq 3$ be any power of an odd prime, $\alpha_1 = 0$, $\alpha_2 = 1$, and $m_1 = 1$.  
In this case, $\mathcal{H}$ is a $(q - 1)$-regular bipartite graph with $q^2$ vertices in each part.  It can be shown that 
$\mathcal{H}$ is isomorphic to a subgraph of the incidence graph of the projective plane $PG(2,q)$.  In particular, 
$\mathcal{H}$ is $C_4$-free.  

\bigskip

In the terminology of forbidden subgraphs, Lemma \ref{super linear} tells us that $\mathcal{H}$ is $C_2$-free.

\begin{lemma}\label{super c3}
If $l = 1$, then the hypergraph $\mathcal{H}$ is $C_3$-free.
\end{lemma}
\begin{proof}
This is certainly true if $r = 2$ as in this case $\mathcal{H}$ is a bipartite graph.  Assume that $r \geq 3$ and suppose 
$\mathcal{H}$ contains a $C_3$.  By Lemma \ref{super linear}, there are three distinct edges 
$e(x_1 , y_1 , a_1, m_1)$, $e(x_2 , y_2 , a_2,m_1)$, and $e(x_3 , y_3 , a_3 , m_1)$ and integers $1 \leq i < j < k \leq r$ such that 
\begin{center}
$(x_1 + \alpha_i ( m_1  a_1 ) , y_1 + \alpha_i ( m_1  a_1^2) , i) = 
( x_2 +  \alpha_i (m_1 a_2 ) , y_2 +  \alpha_i (m_1 a_2^2 ) , i )$,

\medskip

$(x_2 +  \alpha_j  (m_1 a_2 ) , y_2 +  \alpha_j (m_1 a_2^2 ) , j) = 
( x_3 + \alpha_j (m_1 a_3 ) , y_3 +  \alpha_j (m_1 a_3^2 ) , j )$,

\medskip

$(x_3 +  \alpha_k (m_1 a_3 ) , y_3 +  \alpha_k (m_1 a_3^2 ) , k) = 
( x_1 + \alpha_k (m_1 a_1 ) , y_1 + \alpha_k (m_1  a_1^2 ) , k )$.
\end{center}
The first equation represents the vertex in $V_i$ that is the 
unique vertex in the intersection of the edges $e( x_1 , y_1 , a_1 , m_1)$ and 
$e(x_2 , y_2 , a_2 , m_1)$. 
    
By considering the equations coming from the first components, we get 
\begin{eqnarray*}
0 & = & (x_1 - x_2) + (x_2 - x_3) + (x_3 - x_1) \\
& = & m_1 \alpha_i ( a_2 - a_1) + m_1 \alpha_j (a_3 - a_2) + m_1 \alpha_k ( a_1 - a_3).
\end{eqnarray*}
Similarly, the equations from the second components give
\begin{equation*}
0 = m_1 \alpha_i ( a_2^2 - a_1^2  ) + m_1 \alpha_j ( a_3^2 - a_2^2) + m_1 \alpha_k (a_1^2 - a_3^2 ).
\end{equation*}
By Lemma \ref{arithmetic l1} with $\alpha= m_1 \alpha_i$, $\beta = m_1 \alpha_j$, and $\gamma = m_1 \alpha_k$, we have
$a_1 = a_2 = a_3$.  
Since
\[
(x_1 +  \alpha_i (m_1 a_1 ) , y_1 +  \alpha_i (m_1 a_1^2 ) , i) =
 ( x_2 +  \alpha_i (m_1 a_2 ) , y_2 +  \alpha_i  (m_1 a_2^2 ) , i ),
\]  
we obtain $x_1 = x_2$ and $y_1 = y_2$ which gives $e(x_1 , y_1, a_1, m_1) = e(x_2 , y_2 , a_2 , m_1)$, a contradiction.
\end{proof}

\bigskip

For the next sequence of lemmas we will require some additional notation and terminology.  For $1 \leq i \neq j \leq r$, 
let $\mathcal{H} (V_i , V_j )$ be the bipartite graph with parts $V_i$ and $V_j$ where 
$(u , v , i ) \in V_i$ is adjacent to $(u' , v' , j ) \in V_j$ if and only if there is an edge $e \in E( \mathcal{H} )$ such that 
\begin{equation}\label{adjacencies}
\{ (u , v, i ) , (u' , v ' , j) \} \subseteq e.
\end{equation}
An equivalent way of defining adjacencies in $\mathcal{H} (V_i , V_j )$ is to say that 
$(u, v, i )$ is adjacent to $(u' , v ' , j)$ if and only if there are elements 
$x,y \in \mathbb{F}_q$, $a \in \mathbb{F}_q^*$, and an $s \in \{1,2, \dots , l \}$ such that 
\begin{equation}\label{adj 2}
u' = u + m_s ( \alpha_j  - \alpha_i ) a \mbox{~~and~~} v' = v + m_s ( \alpha_j  - \alpha_i ) a^2.
\end{equation}
This is because if (\ref{adjacencies}) holds with $e = e( x, y , a , m_s)$, then 
\begin{center}
$u = x + \alpha_i (m_s a )$, $v = y +  \alpha_i (m_s a^2 )$, 
$u' = x +  \alpha_j (m_s a)$, and $v' = y +  \alpha_j  (m_s a^2 )$.
\end{center}
For three distinct integers $i, j$, and $k$ with $1 \leq i , j , k  \leq r$, let $\mathcal{H} (V_i , V_j , V_k)$ be the 
union of the graphs $\mathcal{H} ( V_i , V_j)$, $\mathcal{H} (V_j , V_k)$, and 
$\mathcal{H} (V_k , V_i )$.  

For any $x,y \in \mathbb{F}_q$ and $a \in \mathbb{F}_q^*$, 
the edge $e(x,y, a , m_s)$ in $\mathcal{H}$ is said to have \emph{color} $m_s$.  An edge $f$ in the graph 
$\mathcal{H} (V_i  ,V_j)$ or 
$\mathcal{H} (V_i  ,V_j , V_k)$ 
is said to have \emph{color} $m_s$ if the unique edge $e$ in $\mathcal{H}$ 
with $f \subseteq e$ has color $m_s$.  The edge $e$ is unique by Lemma \ref{super linear}. 

\begin{lemma}\label{c4 free in one color}
For any $1 \leq i \neq j \leq r$ and $1 \leq s \leq l$, the edges of color $m_s$ in the graph 
$\mathcal{H} (V_i , V_j)$ induce a $K_{2,2}$-free graph.
\end{lemma}
\begin{proof}
Suppose $\{ (u_1 , v_1 , i), (u_2 , v_2 , j), (u_3 , v_3 , i), (u_4 , v_4,j) \}$ forms a $K_{2,2}$ in 
$\mathcal{H} (V_i , V_j)$ where each of the edges of this $K_{2,2}$ have color $m_s$. Using (\ref{adj 2}) as our condition 
for adjacency in $\mathcal{H} (V_i , V_j)$, we have 
\begin{center}
$u_2 = u_1 + m_s ( \alpha_j - \alpha_i) a_1 = u_3  + m_s ( \alpha_j - \alpha_i )a_2$,

\medskip

$v_2= v_1 + m_s ( \alpha_j - \alpha_i ) a_1^2 = v_3 + m_s ( \alpha_j - \alpha_i )a_2^2$,

\medskip

$u_4 = u_1 + m_s ( \alpha_j - \alpha_i ) a_3 = u_3 + m_s ( \alpha_j - \alpha_i ) a_4$,

\medskip

$v_4 = v_1 + m_s ( \alpha_j - \alpha_i ) a_3^2 = v_3 + m_s ( \alpha_j - \alpha_i ) a_4^2$
\end{center}
for some $a_1 , a_2 , a_3 , a_4 \in \mathbb{F}_q^*$.  
By the first and third set of equations, 
\[
m_s^{-1} (\alpha_j - \alpha_i)^{-1} (u_1 - u_3) = a_2 - a_1 = a_4 - a_3.
\]
Similarly, by the second and fourth set of equations, $a_2^2  - a_1^2 = a_4^2 - a_3^2$.  
By Lemma \ref{arithmetic l2}, either 
$(a_1 , a_4) = ( a_2 , a_3)$ or $(a_1 , a_4) = (a_3 , a_2)$.  

If $a_1 = a_2$, then $u_1 = u_3$ by the first set of equations and $v_1 = v_3$ by the second set of equations.  This implies 
$(u_1 , v_1 , i)$ and $(u_3 , v_3 , i)$ are the same vertex which is a contradiction.

If $a_1 = a_3$, then by taking differences of the first and third set of equations we get $u_2 = u_4$.  
By taking differences of the second and fourth set of equations we get $v_2 = v_4$.  This implies that the vertices 
$(u_2, v_2 , j)$ and $(u_4 , v_4 , j)$ are the same which is another contradiction.  
\end{proof}

\begin{lemma}\label{two part case}
If $1 \leq i \neq j \leq r$, then for any $l \geq 1$, the graph $\mathcal{H}(V_i , V_j)$ is $K_{2,2l^2 - l + 1}$-free.
\end{lemma} 
\begin{proof}
If $l = 1$, then we are done by Lemma \ref{c4 free in one color} as all of the 
edges in $\mathcal{H} (V_i , V_j)$ will have the same color, 
namely $m_1$.  

Assume that $l \geq 2$ and 
suppose $u,v , w_1 , \dots , w_{2l^2 - l + 1}$ are the vertices of of $K_{2,2l^2 - l + 1}$ in $\mathcal{H} (V_i , V_j)$ with 
$u, v \in V_i$ and $w_1 , \dots , w_{2l^2 - l + 1} \in V_j$.  
Since $\frac{2l^2 - l + 1}{l} > 2l - 1$, there are at least $2l$ edges of the form $\{u , w_z \}$ that have the same color.  
Without loss of generality, 
assume that for $1 \leq z \leq 2l$, the edges $\{ u , w_z \}$ have color $m_1$.  
Let $W = \{w_1 , \dots , w_{2l } \}$.  By Lemma \ref{c4 free in one color}, there cannot be two distinct edges, 
both with color $m_1$, that are incident with $v$ and a vertex in $W$.  Thus, at least $2l - 1$ of the edges between $W$ and 
$v$ have a color other than $m_1$.  
As $\frac{2l-1}{l-1} > 2$, there must be three edges between $W$ and $v$ that all have the same color.
Without loss of generality, assume that $\{v, w_1 \}$, $\{v , w_2 \}$, and $\{v , w_3 \}$ all have color $m_2$.
Let $v = (x_v , y_v , i )$, $u = (x_u , y_u , i)$, and $w_z = ( x_{ w_z} , y_{w_z} , j)$ for $z \in \{1,2, 3 \}$.  
For each $z \in \{1,2,3 \}$, there are elements $a_z , b_z \in \mathbb{F}_q^*$ with 
\[
x_{w_z} = x_u + m_1 ( \alpha_j - \alpha_i ) a_z = 
x_v + m_2 ( \alpha_j - \alpha_i ) b_z
\]
and 
\[
y_{w_z} = y_u + m_1 ( \alpha_j - \alpha_i ) a_z^2 = 
y_v + m_2 ( \alpha_j - \alpha_i ) b_z^2.
\]
From these equations we obtain  
\begin{eqnarray*}
x_v - x_u = m_1 ( \alpha_j  - \alpha_i )a_1 + m_2 ( \alpha_i - \alpha_j)b_1 & = &  
m_1 ( \alpha_j  - \alpha_i )a_2 + m_2 ( \alpha_i - \alpha_j)b_2  \\
& = & m_1 ( \alpha_j  - \alpha_i )a_3 + m_2 ( \alpha_i - \alpha_j)b_3
\end{eqnarray*}
and
\begin{eqnarray*}
y_v - y_u = m_1 ( \alpha_j  - \alpha_i )a_1^2 + m_2 ( \alpha_i - \alpha_j)b_1^2 & = & 
m_1 ( \alpha_j  - \alpha_i )a_2^2 + m_2 ( \alpha_i - \alpha_j)b_2^2 \\
& = &
m_1 ( \alpha_j  - \alpha_i )a_3^2 + m_2 ( \alpha_i - \alpha_j)b_3^2.
\end{eqnarray*}  
We want to apply Lemma \ref{arithmetic l3} with 
$\alpha = m_1 ( \alpha_j - \alpha_i)$ and  
$\beta = m_2 ( \alpha_i - \alpha_j)$ but before doing so, we verify that we have satisfied the 
hypothesis of Lemma \ref{arithmetic l3}.  Since $m_i \neq 0$, and $\alpha_i - \alpha_j \neq 0$, 
both $\alpha$ and $\beta$ are not zero.  If $\alpha + \beta = 0$, then 
\[
0 = m_1 ( \alpha_j - \alpha_i ) + m_2 ( \alpha_i - \alpha_j ) = \alpha_i ( m_2 - m_1) - \alpha_j (m_2 - m_1)
\]
so $\alpha_i (m_2 - m_1) = \alpha_j ( m_2 - m_1)$.  As $m_1$ and $m_2$ are distinct, $m_2 - m_1 \neq 0$ so 
$\alpha_i  = \alpha_j$ which contradicts the fact that $\alpha_i$ and $\alpha_j$ are distinct. We conclude that 
$\alpha + \beta \neq 0$ and Lemma \ref{arithmetic l3} applies so we may assume that $a_1 = b_1$ and $a_2 = b_2$.  These two equalities together with 
\[
m_1 ( \alpha_j  - \alpha_i )a_1 + m_2 ( \alpha_i - \alpha_j)b_1  =   
m_1 ( \alpha_j  - \alpha_i )a_2 + m_2 ( \alpha_i - \alpha_j)b_2  
\]
give
\[
(m_1 - m_2) ( \alpha_j - \alpha_i ) a_1 = ( m_1 - m_2)( \alpha_j - \alpha_i ) a_2.
\]
Therefore, $a_1 = a_2$.      

From the equations 
\begin{center}
$x_{w_1}  = x_u + m_1 ( \alpha_j - \alpha_i) a_1$ ~ and ~
$x_{w_2} = x_u  + m_1 (\alpha_j - \alpha_i )  a_2$
\end{center}
we get $x_{w_1} = x_{w_2}$.  A similar argument gives $y_{w_1} = y_{w_2}$, thus 
\[
w_1 = ( x_{w_1} , y_{w_1} , j) = (x_{w_2} , y_{w_2} , j ) = w_2
\]
which provides the needed contradiction.  We conclude that $\mathcal{H} ( V_i , V_j )$ is 
$K_{2 , 2l^2 - l + 1}$-free.  
\end{proof}

\begin{lemma}\label{three part case}
Let $i , j$, and $k $ be distinct integers with $1 \leq i ,j , k \leq r$.  For any $l \geq 1$, the graph $\mathcal{H}(V_i , V_j , V_k)$ 
does not contain a $K_{2, 2l^2 + 1}$ with one vertex in $V_i$, one vertex in $V_j$, and $2l^2 + 1$ vertices in $V_k$.  
\end{lemma}
\begin{proof}
We proceed as in the proof of Lemma \ref{two part case}.  Suppose $\{u , v \}$ and $\{w_1 , \dots , w_{2l^2 + 1} \}$ are the parts of the $K_{2, 2l^2 + 1}$ with
$u \in V_i$, $v \in V_j$, and $w_1 , \dots , w_{2l^2 + 1} \in V_k$.  
As $\frac{2l^2 + 1}{l} > 2l$, we can assume that the edges $\{ u , w_1 \} , \dots , \{ u , w_{2l + 1 } \}$ all have the 
same color, say $m_1$.      
Since $\frac{2l+1}{l}  > 2$, we can assume that at least three of the edges $\{ v , w_1 \} , \dots , \{ v , w_{2l + 1} \}$ 
have the same color.
Let $\{v , w_1 \} , \{ v , w_2 \}$, and $\{v , w_3 \}$ have color $m_s$.    
As in the proof of Lemma \ref{two part case}, we have elements 
$a_1 , a_2 , a_3 , b_1 , b_2 , b_3 \in \mathbb{F}_q^*$ such that 
\begin{eqnarray*}
m_1 ( \alpha_k  - \alpha_i )a_1 + m_s ( \alpha_j - \alpha_k)b_1 & = &  
m_1 ( \alpha_k  - \alpha_i )a_2 + m_s ( \alpha_j - \alpha_k)b_2  \\
& = & m_1 ( \alpha_k  - \alpha_i )a_3 + m_s ( \alpha_j - \alpha_k)b_3 ,
\end{eqnarray*}
and
\begin{eqnarray*}
m_1 ( \alpha_k  - \alpha_i )a_1^2 + m_s ( \alpha_j - \alpha_k)b_1^2 & = &  
m_1 ( \alpha_k  - \alpha_i )a_2^2 + m_s ( \alpha_j - \alpha_k)b_2^2  \\
& = & m_1 ( \alpha_k  - \alpha_i )a_3^2 + m_s ( \alpha_j - \alpha_k)b_3^2 .
\end{eqnarray*}  
If $s =1$ (so $m_s = m_1$), then we apply 
Lemma \ref{arithmetic l3} with $\alpha = m_1 ( \alpha_k - \alpha_i )$ and 
$\beta = m_1 ( \alpha_j - \alpha_k)$ noting that $\alpha + \beta = m_1 ( \alpha_j - \alpha_i  ) \neq 0$.  
If $s \neq 1$, then 
without loss of generality, assume that $s = 2$.  We apply Lemma 
\ref{arithmetic l3} with 
\[
\alpha = m_1 ( \alpha_k - \alpha_i) ~ \mbox{and} ~ 
\beta = m_2 ( \alpha_j - \alpha_k).
\]  
Here we recall that by (\ref{condition}), the $m_t$'s have been chosen so that  
$m_1 ( \alpha_k - \alpha_i ) \neq m_2 ( \alpha_k - \alpha_j)$ so $\alpha + \beta \neq 0$.  
In both cases, we can apply Lemma \ref{arithmetic l3} to get $a_1 = b_1$ and $a_2 = b_2$.  
The remainder of the proof is then identical to that of Lemma \ref{two part case}.  
\end{proof}

\bigskip

\begin{proof}[Proof of the lower bound in Theorem \ref{main thm} and Theorem \ref{main thm2}]
Let $r \geq 3$ be an integer and $l = 1$.  
Let $q \geq r$ be a power of an odd prime and $\alpha_1 , \dots , \alpha_r$ be distinct elements of $\mathbb{F}_q$.
Let $m_1 = 1 \in \mathbb{F}_q$ and note that (\ref{condition}) holds for $\alpha_1 , \dots , \alpha _r$ and $m_1$ since 
in this case, (\ref{condition}) is equivalent to the statement that $\alpha_1 , \dots , \alpha_r$ are all different.  
Let $\mathcal{H}$ be the corresponding hypergraph defined at the beginning of Section \ref{the H}.  By 
Lemmas \ref{super linear} and \ref{super c3}, $\mathcal{H}$ is $\{C_2 , C_3 \}$-free.
Now we show that $\mathcal{H}$ is $K_{2 , 2r - 3}$-free.  

Suppose $\{u , v \}$ and $W = \{w_1 , \dots , w_{2r - 3 } \}$ are the parts of a $K_{2 , 2r - 3}$ in $\mathcal{H}$.  
If $\{ u , v \} \subset V_i$ for some $i \in \{1,2, \dots , r \}$, then by Lemma \ref{c4 free in one color},
$| V_j \cap W | \leq 1$ for each $j \in \{1,2, \dots , r \} \backslash \{i \}$.  
This is impossible since $2r - 3 > r - 1$ as $r > 2$.  Now suppose $u \in V_i$ and $v \in V_j$ 
where $1 \leq i < j \leq r$.  By Lemma \ref{three part case}, 
$| V_k \cap W | \leq 2$ for each $k \in \{1,2, \dots ,r \} \backslash \{ i , j \}$.  Once again this is impossible 
since $2r - 3 > 2 (r - 2)$.  This shows that $\mathcal{H}$ is $K_{2 , 2r-3 }$-free.  
The proof is completed by observing that $\mathcal{H}$ has $q^2$ vertices in each part $V_1 , \dots , V_r$ and $\mathcal{H}$ has 
$q^2 (q - 1)$ edges.  
\end{proof}

\bigskip

\begin{proof}[Proof of Theorem \ref{thm3}]
Let $r \geq 3$ and let $l$ be any integer with $2l + 1 \geq r$.  This assumption on $l$ implies that 
\begin{equation}\label{l eq}
(r - 2)(2l^2) \leq (r - 1) ( 2l^2 - l) .
\end{equation}
Let $q$ be a power of an odd prime chosen large enough so that there are $r$ distinct elements 
$\alpha_1 , \dots , \alpha_r \in \mathbb{F}_q$ and $l$ distinct elements $m_1 , \dots , m_l \in \mathbb{F}_q^*$ 
that satisfy condition (\ref{condition}).  
We claim that choosing $q \geq 2 l r^3$ 
is sufficient for such elements to exist.    
Indeed, we first choose $\alpha_1 , \dots , \alpha_r$ so that these elements are all distinct.  We then choose 
the $m_z$'s.  If we have chosen $m_1 , \dots , m_t$ so that (\ref{condition}) holds 
for $\alpha_1 , \dots , \alpha_r$ and $m_1 , \dots , m_t$, then as long as we choose $m_{t+1}$ so that 
$m_{t + 1} \neq m_z ( \alpha_k - \alpha_j) (\alpha_k - \alpha_i)^{-1}$, then 
(\ref{condition}) holds for $\alpha_1 , \dots , \alpha_r$ and $m_1 , \dots , m_t , m_{t + 1}$.  
There are at most $t r^3 $ products of the form $m_z ( \alpha_k - \alpha_j ) ( \alpha_k - \alpha_i )^{-1}$
with $z \in \{1, \dots , t \}$ and $1 \leq i , j , k \leq r$
so $q \geq 2 l r^3$ is enough to choose $m_{t+1}$.  

Having chosen $\alpha_1 , \dots , \alpha_r$ and $m_1 , \dots , m_l$, let $\mathcal{H}$ be the corresponding hypergraph.  
By Lemma \ref{super linear}, $\mathcal{H}$ is $C_2$-free.  Now we show that $\mathcal{H}$ is 
$K_{2 , ( r - 1)( 2l^2 - l ) + 1 }$-free.  

Suppose $\{u,v \}$ and $W = \{w_1 , \dots , w_t \}$ are the parts of a $K_{2 ,t}$ in $\mathcal{H}$.  
If $\{u , v \} \subset V_i$ for some $i$, then by Lemma \ref{two part case}, 
$| V_j \cap W | \leq 2l^2 - l$ for each $j \in \{1,2, \dots , r \} \backslash \{ i \}$ so $t \leq (r - 1) ( 2l^2 - l )$.  
If $u \in V_i$ and $v \in V_j$ for some $1 \leq i < j \leq r$, then 
by Lemma \ref{three part case}, $| V_k \cap W | \leq 2l^2 $ for each $k \in \{1,2, \dots , r \} \backslash \{ i ,j  \}$
so $t \leq (r - 2) (2l^2)$ thus by (\ref{l eq}), $t \leq (r - 1)(2l^2 - l)$.  
We conclude that $\mathcal{H}$ is $K_{2 , (r - 1) ( 2l^2 - l )+1}$-free.  The proof of Theorem 
\ref{thm3} is completed by observing that $\mathcal{H}$ has $rq^2$ vertices and $l q^2 ( q - 1)$ edges.  
\end{proof}


\section{Concluding Remarks and Acknowledgments}

It was pointed out to the author by Cory Palmer that the argument used to prove Theorem \ref{ub} can be used to show that 
\[
\textup{ex}_r ( n , \{ C_2 , K_{2,t + 1} \} ) \leq \frac{ \sqrt{ 2 (t + 1) } }{r } n^{3/2} + \frac{n}{r}
\]
for all $r \geq 3$ and $t \geq 1$.  This shows that the lower bound in Theorem \ref{thm3} gives the correct order of magnitude 
but determining the correct constant could be difficult.  It is known that in the case of graphs, 
$\textup{ex}_2 ( n , K_{2 , t + 1} ) = \frac{1}{2} \sqrt{ t} n^{3/2} + o( n^{3/2} )$ (see F\"{u}redi \cite{fur}).

The author would like to thank Cory Palmer and Jacques Verstra\"{e}te for helpful discussions.  


\end{document}